\def\bu{\bullet}
\def\marker{\>\hbox{${\vcenter{\vbox{
    \hrule height 0.4pt\hbox{\vrule width 0.4pt height 6pt
    \kern6pt\vrule width 0.4pt}\hrule height 0.4pt}}}$}\>}
\def\gpic#1{#1
     \medskip\par\noindent{\centerline{\box\graph}} \medskip}
\documentclass[12pt]{article}
\usepackage{amsmath,amsthm,amssymb,verbatim,url}
\usepackage{fullpage}

\newcommand{\vey}[1]{\mathcal{W}^{#1}}
\newcommand{\ved}{\vey{d}}

\newcommand{\ints}{\mathbb{Z}}
\newcommand{\NN}{\mathbb{N}}

\newtheorem{theorem}{Theorem}[section]

\newtheorem{lemma}[theorem]{Lemma}
\newtheorem{conj}[theorem]{Conjecture}
\newtheorem{prop}[theorem]{Proposition}

\newtheorem{corollary}[theorem]{Corollary}
\newtheorem*{innerptlanswer}{Partial Answer to Question \anslabel}
\newtheorem*{inneranswer}{Answer to Question \anslabel}

\theoremstyle{definition}
\newtheorem{definition}[theorem]{Definition}

\newcommand{\anslabel}{foo}

\newcommand{\sizeof}[1]{\left|{#1}\right|}
\newcommand{\FL}[1]{\left\lfloor #1 \right\rfloor}
\newcommand{\CL}[1]{\left\lceil #1 \right\rceil}
\def\H{\mathcal{H}}
\def\VEC#1#2#3{#1_{#2},\ldots,#1_{#3}}
\def\B#1{{\bf #1}}
\def\SE#1#2#3{\sum_{#1=#2}^{#3}}
\def\PE#1#2#3{\prod_{#1=#2}^#3}
\def\st{\colon\,}
\def\CH#1#2{\binom{#1}{#2}}
\def\P{{\rm Pr}}
\def\C#1{\left|{#1}\right|}
\def\plu{\pmb+}
\def\Bkpsq{\B{\pmb(k\plu1\pmb)}^2}

\title{Chain-making games in grid-like posets}
\author{
Daniel W. Cranston\thanks{Mathematics Dept., Virginia Commonwealth University,
dcranston@vcu.edu},
William B. Kinnersley\thanks{Mathematics Dept., University of Illinois,
wkinner2@illinois.edu},
Kevin G. Milans\thanks{Mathematics Dept., University of South Carolina,
milans@math.sc.edu},\\
Gregory J. Puleo\thanks{Mathematics Dept., University of Illinois,
puleo@illinois.edu.  Research of Kinnersley, Milans, and Puleo supported by
NSF grant DMS 08-38434, ``EMSW21-MCTP: Research Experience for Graduate
Students''.},
and Douglas B. West\thanks{Mathematics Dept., University of Illinois,
west@math.uiuc.edu.  Research partially supported by
the National Security Agency under Award No.~H98230-10-1-0363.}
}

\begin{document}
\vspace{-.2in}
\maketitle

\vspace{-1pc}

\begin{abstract}
We study the Maker-Breaker game on the hypergraph of chains of fixed size in a
poset.  In a product of chains, the maximum size of a chain that Maker can
guarantee building is $k-\FL{r/2}$, where $k$ is the maximum size of a chain in
the product, and $r$ is the maximum size of a factor chain.  We also study a
variant in which Maker must follow the chain in order, called the
{\it Walker-Blocker game}.  In the poset consisting of the bottom $k$ levels
of the product of $d$ arbitrarily long chains, Walker can guarantee a chain
that hits all levels if $d\ge14$; this result uses a solution to Conway's
Angel-Devil game.  When $d=2$, the maximum that Walker can guarantee is only
$2/3$ of the levels, and $2/3$ is asymptotically achievable in the product of
two equal chains.
\end{abstract}

\section{Introduction}
The {\it Maker-Breaker} game on a hypergraph $\H$ is played by Maker and
Breaker, who alternate turns (beginning with Maker).  Each player moves by
choosing a previously unchosen vertex of $\H$.  Maker wins by acquiring all
vertices of some edge of $\H$; Breaker wins if all vertices are chosen without
Maker winning.

Maker-Breaker games have been studied for many hypergraphs, particularly when
the vertices are the edges of an $n$-vertex complete graph.  In that setting,
when $n$ is large enough Maker can build a spanning cycle, a complete subgraph
with $q$ vertices, a spanning $k$-connected subgraph, or various other
structures.  For an introduction to Maker-Breaker games (and more general
positional games), see the surveys~\cite{Beck1} and~\cite{Beck2} or the
book~\cite{Beck3}.  Recent work has concentrated on finding ``efficient
strategies'', winning the game as quickly as possible (see~\cite{BL, HKSS}).

In this paper, we study the {\it chain game} on posets, where the winning sets
are the chains with a given size.  For every poset there is a maximum size of
chain that Maker can build against optimal play by Breaker; we seek this value.
For special posets whose elements are integer $d$-tuples, we give efficient
strategies for Maker that do not waste any move (every element that Maker
selects is in the chain constructed).

Chains in posets are ordered from bottom to top, so a natural variant of the
chain game is the {\it ordered chain game}, in which the chain must be built
from bottom to top.  To distinguish this from the (unordered) chain game, we
call its players {\it Walker} and {\it Blocker}.

The {\it product} of $d$ chains with sizes $\VEC r1d$, written
$\PE i1d \B r_i$, is the set of $d$-tuples $x$ such that $0\le x_i< r_i$ for
$1\le i\le d$, ordered by $x\preceq y$ if and only if $x_i\le y_i$ for
$1\le i\le d$.  The {\it $d$-dimensional $k$-wedge}, written $\vey{d}_k$, is
the subposet of $\B k^d$ consisting of the nonnegative-integer $d$-tuples with
sum less than $k$, under the coordinate-wise order.

For the chain game and ordered chain game on these posets, we prove two main
results.

\begin{theorem}\label{chainprod}
Let $P$ be a product of $d$ chains, with $r$ being the maximum size among these
chains and $k$ being the maximum size of a chain in $P$.  In the chain game on
$P$, Maker can build a chain of size $k-\FL{r/2}$, and Breaker can prevent
Maker from building a larger chain.
\end{theorem}

\begin{theorem}\label{wedgethm}
If $d\ge 14$, and $k\in\NN$, then in the ordered chain game on $\vey{d}_k$,
Walker can build a chain of size $k$ (hitting all levels).
\end{theorem}

In small dimensions, Walker cannot guarantee hitting all levels.
In particular, when $d=2$ Walker can only get two-thirds of the levels.
Somewhere between dimension 2 and dimension 14 there is a ``phase transition''
after which Walker can hit all levels.  We do not know the value where this
occurs, but we conjecture that Walker wins already when $d=3$.

We begin in Section~\ref{chainprod} with the Maker-Breaker chain game on
products of chains.  The remainder of the paper addresses the Walker-Blocker
game.  In Section~\ref{d=2} we study the $2$-dimensional case for both wedges
and grids (products of equal chains).  For $\vey{2}_k$, the maximum size of a
chain that Walker can guarantee building is $\CL{2k/3}$.  In the product of
two chains of equal size, which is contained in a $2$-dimensional wedge, Walker
can still guarantee asymptotically $2/3$ of the levels; we prove this using a
``potential function'' argument.  Section~\ref{angel} relates the Walker-Blocker
game to Conway's famous Angel-Devil game.  We apply this relationship in
Section~\ref{wedge} to prove Theorem~\ref{wedgethm}.  We conjecture that
the conclusion of Theorem~\ref{wedgethm} holds in fact for $d\ge3$.  Finally,
Section~\ref{bias} addresses the biased game in which Blocker makes $b$ moves
after each move by Walker.

\section{Maker-Breaker on Chain-products}\label{chainprod}
In a product of chains, we use {\it level $\ell$} to denote the set of elements
whose entries sum to $\ell$.  A {\it successor} of a $d$-tuple $x$ is a
$d$-tuple $y$ such that $x\prec y$.  To evoke familiar terminology from games
on physical boards, we refer to an element chosen at a particular time as a
{\it move} and say that the player {\it plays} that move at that time.

In order to solve the Maker-Breaker game on products of finite chains,
we first solve the Walker-Blocker game on products of $2$-element chains.
We then apply this lemma to build an optimal strategy for Maker in the
unordered chain game on arbitrary finite chain-products.  Let
$[d]=\{1,\ldots,d\}$.  There is a natural isomorphism from $\B 2^d$ to the
lattice of subsets of $[d]$ in which each binary $d$-tuple $x$ is mapped to
$\{i\st x_i=1\}$.

\begin{lemma}\label{hypercube}
For $d\ge2$, let $P'$ be the poset obtained from $\B 2^d$ by deleting the top
element and the bottom element.  The maximum size of a chain in $P'$ is $d-1$,
and Walker can build a chain of size $d-1$ in $P'$, even if Blocker moves first.
\end{lemma}
\begin{proof}
We refer to the elements of $P'$ by the corresponding subsets in
$\{1,\ldots,d\}$.  On his $k$th turn, Walker plays a move $S$ at level $k$ such
that (i) $S$ is above all his previous moves and (ii) Blocker has played no
successor of $S$.  Successfully executing this strategy for $d-1$ turns builds
a chain of size $d-1$.  

Let $S$ be the previous move by Walker.  If Blocker responded with a move not
a successor of $S$, then Walker can add any element to $S$.  Otherwise, since
the highest level of $P'$ is $d-1$, the move by Blocker omits some $e\in [d]$.
Walker now plays $S\cup\{e\}$ and restores the property that no successor of
the current move has been played.
\end{proof}

Since Walker can build a chain hitting all levels in $P'$, we conclude also
that Maker can build a chain hitting all levels in the unordered game.  The
latter statement, along with the freedom to let Breaker move first, is what
we need to analyze arbitrary chain-products.  Since Maker need not take the
elements of a chain in order, Maker can build chains independently in different
copies of the poset $P'$ in Lemma~\ref{hypercube}; they will combine to form a
large chain.

To show optimality of the resulting strategy for Maker, we present a strategy
for Breaker.  Since every chain in a chain-product is contained in a longest
chain, it suffices to give a pairing strategy for Breaker that guarantees
blocking enough of every longest chain.  

\begin{theorem} \label{PoC:upper} \label{PoC:lower}
Let $P=\PE i1d \B r_i$.  In the unordered chain game on $P$ with $r=\max_i r_i$,
Maker can guarantee building a chain of size $k-\FL{r/2}$, where $k$ is the
maximum size of a chain in $P$, and Breaker can keep Maker from building any
larger chain.
\end{theorem}
\begin{proof}
The elements of $P$ are the $d$-tuples $x$ such that $0\le x_i<r_i$ for
$1\le i\le d$.  By symmetry, we may assume that $r=r_d=\max_i r_i$.  For
$0\le j< r$, let $z_j$ be the element of $P$ whose $i$th coordinate is
$\min(j,r_i-1)$.  For $1\le j< r$, let $A_j$ be the subposet of $P$ consisting
of all $x$ such that $z_{j-1}\prec x\prec z_j$.  Note that
$z_0\prec \cdots\prec z_{r-1}$, that $\VEC A1{r-1}$ are pairwise disjoint, and
that each $A_j$ is isomorphic to the poset $P'$ of Lemma~\ref{hypercube} for
some dimension (as $j$ increases beyond some $r_i$, the dimension decreases).
Fig.~\ref{subcubes} illustrates $\VEC A1{r-1}$.

The key to Maker's strategy is that chains in $\VEC A1{r-1}$ combine to form a
longer chain in $P$.  Let $Z = \{\VEC z0{r-1}\}$; these are the bold elements
in Fig.~\ref{subcubes}.  Maker begins by playing an element in $Z$.  When
Breaker plays an element in $Z$, Maker responds by playing another element in
$Z$ if one is available.  Maker treats each subposet $A_j$ as an instance of
the poset $P'$ of Lemma~\ref{hypercube}; when Breaker plays in $A_j$, Maker
responds using the strategy of Lemma~\ref{hypercube}.  When Breaker plays any
other move, Maker plays to increase the chain in some $A_j$ or $Z$.

By Lemma~\ref{hypercube}, Maker obtains in each $A_j$ a chain hitting all 
levels.  These combine with $Z$ to form a long chain in $P$.  The only levels
that Maker misses are those containing an element of $Z$ played by Breaker.
Maker's strategy ensures that Breaker plays at most $\FL{|Z|/2}$ such moves.
Since $|Z|=r$, the bound follows.

To prove optimality, Breaker uses a pairing strategy.  For all $j$ with
$0\le j<\FL{r/2}$ and all $(\VEC x1{d-1})$, Breaker pairs the element
$(\VEC x1{d-1},2j)$ with the element $(\VEC x1{d-1},2j+1)$.  When Maker plays a
paired element, Breaker plays its mate; when Maker plays an unpaired element,
Breaker responds arbitrarily.  We show that Maker misses at least $\FL{r/2}$
elements from every maximal chain.

Given a maximal chain $X$, let $X_j$ be the subchain of $X$ consisting of all
elements whose last coordinate has value $j$, for $0\le j\le r-1$.  Let $x$ be
the last element of $X$ in $X_{2m}$, and let $y$ be the first element of 
$X$ in $X_{2m+1}$ (note that $X$ cannot skip any $X_j$).  Since Breaker has
paired $x$ with $y$, Maker misses at least one of these.  Because there are
$\FL{r/2}$ even integers that are at least $0$ and less than $r-1$, the theorem
follows.
\end{proof}

\vspace{-.5pc}
\begin{figure}[hbt]
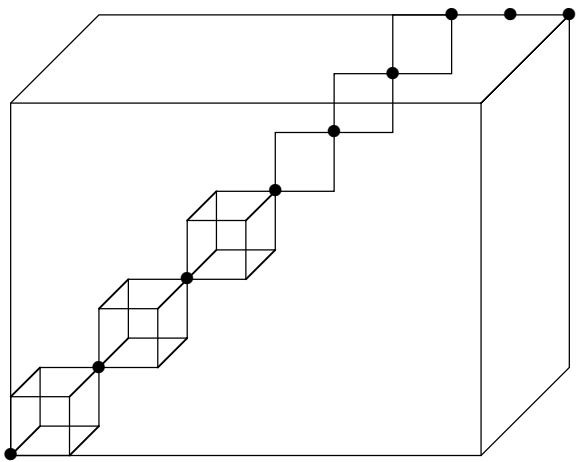

\def\bu{{\Huge \bullet}}
\gpic{
\expandafter\ifx\csname graph\endcsname\relax \csname newbox\endcsname\graph\fi
\expandafter\ifx\csname graphtemp\endcsname\relax \csname newdimen\endcsname\graphtemp\fi
\setbox\graph=\vtop{\vskip 0pt\hbox{%
    \special{pn 8}%
    \special{pa 38 2346}%
    \special{pa 2500 2346}%
    \special{pa 2962 1885}%
    \special{pa 2962 38}%
    \special{pa 500 38}%
    \special{pa 38 500}%
    \special{pa 38 2346}%
    \special{fp}%
    \special{pa 2500 2346}%
    \special{pa 2500 500}%
    \special{pa 2962 38}%
    \special{pa 2500 500}%
    \special{pa 38 500}%
    \special{fp}%
    \special{pa 38 2346}%
    \special{pa 346 2346}%
    \special{pa 346 2038}%
    \special{pa 38 2038}%
    \special{pa 38 2346}%
    \special{fp}%
    \special{pa 192 2192}%
    \special{pa 500 2192}%
    \special{pa 500 1885}%
    \special{pa 192 1885}%
    \special{pa 192 2192}%
    \special{fp}%
    \special{pa 500 1885}%
    \special{pa 808 1885}%
    \special{pa 808 1577}%
    \special{pa 500 1577}%
    \special{pa 500 1885}%
    \special{fp}%
    \special{pa 654 1731}%
    \special{pa 962 1731}%
    \special{pa 962 1423}%
    \special{pa 654 1423}%
    \special{pa 654 1731}%
    \special{fp}%
    \special{pa 962 1423}%
    \special{pa 1269 1423}%
    \special{pa 1269 1115}%
    \special{pa 962 1115}%
    \special{pa 962 1423}%
    \special{fp}%
    \special{pa 1115 1269}%
    \special{pa 1423 1269}%
    \special{pa 1423 962}%
    \special{pa 1115 962}%
    \special{pa 1115 1269}%
    \special{fp}%
    \special{pa 1423 962}%
    \special{pa 1731 962}%
    \special{pa 1731 654}%
    \special{pa 1423 654}%
    \special{pa 1423 962}%
    \special{fp}%
    \special{pa 1731 654}%
    \special{pa 2038 654}%
    \special{pa 2038 346}%
    \special{pa 1731 346}%
    \special{pa 1731 654}%
    \special{fp}%
    \special{pa 2038 346}%
    \special{pa 2346 346}%
    \special{pa 2346 38}%
    \special{pa 2038 38}%
    \special{pa 2038 346}%
    \special{fp}%
    \special{pn 11}%
    \special{pa 38 2346}%
    \special{pa 192 2192}%
    \special{fp}%
    \special{pa 346 2346}%
    \special{pa 500 2192}%
    \special{fp}%
    \special{pa 346 2038}%
    \special{pa 500 1885}%
    \special{fp}%
    \special{pa 38 2038}%
    \special{pa 192 1885}%
    \special{fp}%
    \special{pa 500 1885}%
    \special{pa 654 1731}%
    \special{fp}%
    \special{pa 808 1885}%
    \special{pa 962 1731}%
    \special{fp}%
    \special{pa 808 1577}%
    \special{pa 962 1423}%
    \special{fp}%
    \special{pa 500 1577}%
    \special{pa 654 1423}%
    \special{fp}%
    \special{pa 962 1423}%
    \special{pa 1115 1269}%
    \special{fp}%
    \special{pa 1269 1423}%
    \special{pa 1423 1269}%
    \special{fp}%
    \special{pa 1269 1115}%
    \special{pa 1423 962}%
    \special{fp}%
    \special{pa 962 1115}%
    \special{pa 1115 962}%
    \special{fp}%
    \graphtemp=.5ex\advance\graphtemp by 2.346in
    \rlap{\kern 0.038in\lower\graphtemp\hbox to 0pt{\hss $\bu$\hss}}%
    \graphtemp=.5ex\advance\graphtemp by 1.885in
    \rlap{\kern 0.500in\lower\graphtemp\hbox to 0pt{\hss $\bu$\hss}}%
    \graphtemp=.5ex\advance\graphtemp by 1.423in
    \rlap{\kern 0.962in\lower\graphtemp\hbox to 0pt{\hss $\bu$\hss}}%
    \graphtemp=.5ex\advance\graphtemp by 0.962in
    \rlap{\kern 1.423in\lower\graphtemp\hbox to 0pt{\hss $\bu$\hss}}%
    \graphtemp=.5ex\advance\graphtemp by 0.654in
    \rlap{\kern 1.731in\lower\graphtemp\hbox to 0pt{\hss $\bu$\hss}}%
    \graphtemp=.5ex\advance\graphtemp by 0.346in
    \rlap{\kern 2.038in\lower\graphtemp\hbox to 0pt{\hss $\bu$\hss}}%
    \graphtemp=.5ex\advance\graphtemp by 0.038in
    \rlap{\kern 2.346in\lower\graphtemp\hbox to 0pt{\hss $\bu$\hss}}%
    \graphtemp=.5ex\advance\graphtemp by 0.038in
    \rlap{\kern 2.654in\lower\graphtemp\hbox to 0pt{\hss $\bu$\hss}}%
    \graphtemp=.5ex\advance\graphtemp by 0.038in
    \rlap{\kern 2.962in\lower\graphtemp\hbox to 0pt{\hss $\bu$\hss}}%
    \hbox{\vrule depth2.385in width0pt height 0pt}%
    \kern 3.000in
  }%
}%
}
\vspace{-1pc}
\caption{Maker strategy for chain-products}\label{subcubes}
\end{figure}

\section{Walker-Blocker in Two Dimensions}\label{d=2}
For the Walker-Blocker game, wedges are easier to analyze than chain-products,
because the poset formed by the successors of any element is isomorphic to a
truncation of the same wedge by discarding the highest levels.
This allows Walker to define a strategy locally.  In a chain-product, when the
top of a growing chain reaches the maximum in a given coordinate, no further
moves in that direction are possible.  Breaker may be able to take advantage
of Walker being ``trapped in a corner''.  To overcome this, Walker may need
to look farther ahead to plan a strategy.

In this section, we first give an exact solution for the Walker-Blocker game
on $\vey{2}_k$.  We then express Walker's strategy in a more general way using
a ``potential function'' to show that asymptotically as big a chain can be
built in the game on the product of two $(k+1)$-element chains as on 
the wedge $\vey{2}_{2k+1}$ that contains it.

\begin{theorem} \label{wedge2}
In the ordered chain game on $\vey{2}_k$, Walker can build a chain of size
$\CL{2k/3}$ in $\CL{2k/3}$ moves.  Also, Blocker can prevent Walker from
building a larger chain.
\end{theorem}
\begin{proof}
We present a strategy in which Walker follows a single chain, with no wasted
moves.  Walker first plays $(0,0)$.  For each subsequent move, Walker plays a
successor of his previous move; among all unchosen successors, he plays one
at the lowest level.  A level containing a move by Walker is ``won by Walker''.
After a move $(a,b)$, Walker next plays at level $a+b+1$ unless Blocker has
already played both $(a+1,b)$ and $(a,b+1)$.  We then say that Blocker wins
level $a+b+1$.

Since Blocker spends at least two moves in a level to win it, while Walker
spends only one move per level won, the number of levels that have been won
by Walker is always at least twice the number won by Blocker.  At the end of
the game all $k$ levels have been won by one player or the other; hence
Walker has won at least $\CL{2k/3}$ levels.

For the upper bound, we present a strategy for Blocker to keep Walker from
building a larger chain.  If Walker's previous move was at $(a,b)$ and
exactly one of $(a+1,b)$ and $(a,b+1)$ have been played, then Blocker plays
the other.  If neither of them has been played, then Blocker plays $(a+1,b+1)$,
if available.  Otherwise, Blocker plays an arbitrary move.

Once the game has finished, let $(a,b)$ be an element on a largest chain $C$
that was occupied by Walker in order.  If when Walker played $(a,b)$ one of
$(a+1,b)$ and $(a,b+1)$ had already been played, then $C$ has no element from
level $a+b+1$.  If $C$ has an element $x$ from level $a+b+1$, then before $x$
was played the element $(a+1,b+1)$ was played by one of the players.  Blocker
next ensures that the other successor of $x$ at level $a+b+2$ is occupied, thus
preventing $C$ from having an element from level $a+b+2$.

We have shown that Blocker's strategy prevents Walker from building a chain
in order that hits three consecutive levels.  Hence Walker wins at most
$\CL{2k/3}$ levels.
\end{proof}

In efficient strategies, where all moves by Walker form a chain, we refer to
the most recent move by Walker as the {\it head} of the chain.  A more global
view of the strategy for Walker uses a potential function to measure the
difficulty that Walker faces in the levels above the head.

We define a potential function to measure the future levels that Walker may
need to skip.  Thus Walker wants to keep the potential small.  When Walker
skips a level, the potential will decrease by 1.  Other moves by Walker will
not increase the potential.  A move by Blocker will increase the potential by
at most $1/2$.  We design such a potential and strategy for $\vey{2}_{2k+1}$
and use it to show that even when the game is restricted to the subposet
$\Bkpsq$, Walker can still win asymptotically $2/3$ of the levels.

Blocker's move at a position $(c,d)$ makes it harder for Walker to win level
$c+d$.  To measure this difficulty when the head is at $(a,b)$, we define the
{\it influence} of $(c,d)$ on $(a,b)$, where $a\le c$ and $b\le d$, to be
$\CH{c'+d'}{d'}2^{-(c'+d')}$, where $(c',d')=(c,d)-(a,b)$ (the influence is
$0$ if $a>c$ or $b>d$).  We write $f_{a,b}(c,d)$ for the influence of $(c,d)$
on $(a,b)$.  Define the {\it potential} $\Phi_{a,b}$ to be the total influence
on $(a,b)$ of the moves Blocker has played.  Large potential is good for
Blocker.

To motivate these definitions, note that the influence of $(c,d)$ on $(a,b)$
equals the probability that a random walk from $(a,b)$ to level $c+d$ will end
at position $(c,d)$, where the walk iteratively increases a randomly chosen
coordinate by 1.  Let $(a,b)$ be the current head of the chain, and let $(c,d)$
be another position.  The average of the influences of $(c,d)$ on $(a+1,b)$
and $(a,b+1)$ equals the influence of $(c,d)$ on $(a,b)$.  Walker will want to
choose the option that produces the smaller potential.
 
\begin{theorem}
\label{mainthm}
In the ordered chain game on $\Bkpsq$, Walker can build a chain hitting
more than $\frac23(2k+1)-4\sqrt{k\ln k}$ of the $2k+1$ levels, and this is
asymptotically sharp.
\end{theorem}
\begin{proof}
Since Blocker can limit Walker to winning $\CL{(2/3)(2k+1)}$ levels in 
$\vey{2}_{2k+1}$, Walker can do no better on the subposet $\Bkpsq$.
Hence it suffices to prove the lower bound.

Consider the game on $\vey{2}_{2k+1}$.  At a given time, let $S(a,b)$ denote
the set of elements at or above $(a,b)$ that Blocker has played.  Recall that
the potential $\Phi_{a,b}$ at a point $(a,b)$ is 
$\sum_{(c,d)\in S(a,b)}f_{a,b}(c,d)$.

We have noted that always
$f_{a,b}(c,d)=\frac{1}{2}\left[f_{a+1,b}(c,d) + f_{a,b+1}(c,d)\right]$.
When the head of the chain is at $(a,b)$, we have $(a,b)\notin S(a,b)$, and
hence summing over $S(a,b)$ yields
$\Phi_{(a,b)} = \frac12(\Phi_{(a+1,b)} + \Phi_{(a,b+1)})$.
To keep the potential small, Walker wants to move to whichever of $(a+1,b)$ and
$(a,b+1)$ has smaller potential.

If this strategy directs Walker to play a move $(a',b')$ that Blocker already
played (that is, $(a',b')\in S(a,b)$), then Walker simply computes the choice
the strategy would make from $(a',b')$ instead.  The influence of $(a',b')$ on
$\Phi_{a',b'}$ is $1$, and by skipping $(a',b')$ this influence is lost.  Thus
when Walker chooses a successor of $(a',b')$, the potential decreases by (at
least) 1.  Walker may skip several levels before the preferred option is
available, decreasing the potential by $1$ for each level skipped.

Since Blocker cannot play where Walker just played, the increase in potential
from Blocker's move is at most $1/2$.  Since the potential is $0$ at the start
and the end of the game, Blocker must make at least two moves for every level
skipped by Walker.  Walker wins a level for each move played, so Walker wins at
least twice as many levels as are skipped.

In order to restrict play to $\Bkpsq$, which is a subposet of
$\vey{2}_{2k+1}$, we grant Blocker initially all moves that are outside
$\Bkpsq$.  The key observation, which we formalize below, is that all of
these free moves for Blocker increase the potential only by $o(k)$.  Since
again the potential decreases to $0$ at the end, Walker loses at most $o(k)$
more levels than in the original game on $\vey{2}_{2k+1}$, and hence Walker
gets at least the fraction $2/3 - o(1)$ of the $2k+1$ levels.  (Here $o(g(k))$
denotes any function of $k$ whose ratio to $g(k)$ tends to $0$ as $k\to\infty$.)

To bound the initial influence of the forbidden moves, recall that $f(c,d)$ is
the probability that a random walk from $(0,0)$ to level $c+d$ ends at $(c,d)$.
The distribution of the endpoint is the standard binomial distribution for
$c+d$ trials.  Let $X_n$ be the binomial random variable counting the heads in
$n$ flips of a fair coin.  The initial value of the potential function is
$\SE n{k+1}{2k} \P(\C{X_n-n/2}> k-n/2)$.  Let $k'=\CL{\sqrt{2k\ln k}}$ and
$n_0=2k-k'$.  For each $n$ greater than $n_0$, the probability is at most $1$.
For $n\le n_0$, we use the well-known Chernoff bound.

The Chernoff bound states that $\P(\C{X_n-n/2} > nt)\le 2e^{-2nt^2}$; we apply
it with $t=k/n-1/2$.  Since $2nt^2$ increases as $n$ decreases, we may
assume $n=n_0$ and use $2e^{-2n_0t^2}$ as a bound on the contribution from
these terms.  We have $n_0=2k(1-x)$, where $x=k'/2k>\sqrt{\ln k/(2k)}$.  Also,
$2t=\frac1{1-x}-1=\frac x{1-x}$.  We compute
$$
2n_0t^2=\frac{n_0}2(2t)^2= k(1-x)\frac{x^2}{(1-x)^2} > kx^2 > \frac12 \ln k.
$$
Thus $2e^{-2n_0t^2}<2k^{-1/2}$, bounding the total contribution from these
terms by $2\sqrt k$.  From the $k'$ terms with largest $n$, the bound on the
total is $\CL{\sqrt{2k\ln k}}$.  Hence the initial potential is less than
$4\sqrt{k\ln k}$.  As a result, Walker misses fewer than $4\sqrt{k\ln k}$
levels in addition to the $(1/3)(2k+1)$ levels missed by the earlier argument.
\end{proof}

\section{Angel-Devil Games on Digraphs}\label{angel}
In this section, we define a slightly more general version of Conway's famous
Angel-Devil game~\cite{Conway}.  We show in the next section that this game is
closely related to the Walker-Blocker game on $\vey{d}_k$, and we will
translate known results about the Angel-Devil game to apply there.

A \emph{rooted digraph} is a digraph $G$, possibly with loops, with one vertex
designated as the \emph{root} or \emph{start vertex}.
The \emph{Angel-Devil game} is played on an infinite rooted digraph $G$ by two
players, Angel and Devil.  Angel and Devil alternate turns, with Angel moving
first.  Each vertex of $G$ is either \emph{burned} or \emph{unburned}, with all
vertices initially unburned.  Angel starts at the root of $G$.  At each turn,
Angel moves to an unburned out-neighbor of his current position.  At each turn,
Devil \emph{burns} at most one vertex, forever denying Angel its use.

Devil wins if Angel is ever unable to move (when all out-neighbors of his
current position are burned).  Angel wins by having a strategy to move forever.
When every vertex of $G$ has finite outdegree, an equivalent statement of the
victory condition for Angel is that for every natural number $n$, Angel has a
strategy to guarantee moving for $n$ turns.
\smallskip

A poset is {\it graded} if all maximal chains joining any two elements have the
same length; like those we have discussed, all graded posets have well-defined
levels.  A {\it rooted poset} is a graded poset with a unique minimal element
$x_0$ called the {\it root}.  A {\it $k$-prefix} is a chain of size $k$
consisting of elements $\VEC x0{k-1}$ such that $x_0$ is the root and $x_i$
covers $x_{i-1}$ for $1\le i\le k$.  The top element of a prefix is its
{\it head}; {\it climbing} a prefix means following it in order.

The \emph{$k$-prefix game} on a rooted poset $P$ is the Walker-Blocker game in
which Walker must climb a $k$-prefix to win.  If Walker wins the $k$-prefix
game on $P$, then Walker also wins the ordered $k$-chain game, since a
$k$-prefix is a $k$-chain with the additional requirements of skipping no
levels and starting at the bottom.  We say that Blocker wins the \emph{prefix
game} on $P$ (without a specified parameter) if Blocker wins the $k$-prefix
game on $P$ for some $k$.

We prove that the prefix game on a rooted poset $P$ is equivalent to the
Angel-Devil game on the rooted cover digraph of $P$.  This is not immediately
obvious because, unlike Angel, Walker can backtrack and take an alternative
climbing route when blocked.  Thus Walker is more powerful than Angel in the
corresponding game, and it is easy to obtain a winning strategy for Walker from
a winning strategy for Angel.

\begin{theorem}\label{MB-AD_equiv}
Let $P$ be a rooted poset with minimal element $x_0$, and assume that every
element of $P$ is covered by finitely many elements.  Let $G$ be the rooted
digraph with start vertex $x_0$ that is the cover digraph of $P$.  Blocker wins
the prefix game on $P$ if and only if Devil wins the Angel-Devil game on $G$.
\end{theorem}
\begin{proof}
Walker can copy a winning Angel strategy.  Both begin at $x_0$ and thereafter
remain at corresponding vertices.  Walker can treat a move by Blocker as a move
by Devil in the Angel-Devil game, using Angel's response as a Walker move in
the prefix game.  This keeps Walker at the same vertex as Angel, so the process
continues.

To prove the converse, we obtain a winning strategy for Blocker from a winning
strategy for Devil.  Imagine Walker as starting with an infinite stack of
Angels at $x_0$.  Walker will maintain having an infinite stack of Angels at
each element of each prefix that has been climbed.  When Walker extends a
prefix, he splits the stack at the previous head into two infinite stacks and
sends one to the new head.  Blocker examines the history of where the Angels in
the new stack have been and responds as Devil would if a lone Angel had
followed that path.  Since the stack moved along that path, all vertices Devil
needed to defend against that Angel have been burned, and hence Blocker has
played them all.  Hence Blocker always has all the moves needed to block all
prefixes started by Walker.

However, Walker may play a move that extends more than one prefix.  The
coalescing stacks have different histories.  Different moves may be needed to
maintain blocking those different Angels, but Blocker can play only one of them.
It suffices for Blocker to pick any one of the Angel histories that reach the
new head and copy the Devil's move to block that Angel, absorbing the coalescing
stacks into that one stack.  The reason is that all those Angels are now in the
same position and move together.  All moves needed to block any one of them via
the Devil strategy are in place, so Devil/Blocker can continue blocking the
Angel with the chosen history wherever it goes.  The moves that have been
played to block Angels on the other paths reaching this position are bonuses
for Blocker and can be ignored.

This strategy may direct Blocker to play a move that has already been played
by Blocker or Walker; in either case, Blocker can play arbitrarily.  When
Blocker is directed to play a position $x$ already played by Walker, the moves
are already in place to block all prefixes using $x$, so Blocker has no need
for $x$.

Since Devil has a winning strategy and Blocker can employ it against Angels
sitting at all heads of prefixes, simultaneously, Walker cannot play
arbitrarily long prefixes.
\end{proof}

It is difficult to devise winning strategies in Angel-Devil games.  To benefit
from the few explicit strategies that are known, we seek ways to transfer these
strategies between games.  We define a type of map from one rooted digraph to
another that facilitates such a transfer.
\begin{definition}
Let $G$ and $H$ be digraphs with roots $g_0$ and $h_0$.  A \emph{robust map}
from $G$ to $H$ is a map $\phi\st V(G) \to V(H)$ with $\phi(g_0) = h_0$ such
that whenever there is an edge from $\phi(v)$ to $w$ in $H$, there is also some
vertex $z$ in $G$ such that $\phi(z)=w$ and $vz\in E(G)$.
\end{definition}

Informally, a map is robust if, whenever the image $\phi(P)$ in $H$ of a path
$P$ in $G$ can be extended, $P$ can also be extended to $P'$ in $G$ so that
$\phi(P')$ is the extended path in $H$.

\begin{theorem}\label{thm:mapwin}
Let $G$ and $H$ be two rooted digraphs, and let $\phi\st G \to H$ be a robust
map from $G$ to $H$.  If Angel wins the Angel-Devil game in $H$, then Angel
also wins in $G$.
\end{theorem}
\begin{proof}
Given a winning strategy for Angel in $H$, we define a winning strategy in $G$.
We play an imaginary game in $H$ to track and simulate the actual game in $G$.
The $G$-Angel will maintain a position in $G$ that maps under $\phi$ to the
current position of the $H$-Angel in $H$.  This holds initially, since they
both start at the root.

At some time later, let $v$ be the location of the $G$-Angel, so the imagined
$H$-Angel is at $\phi(v)$.  The $G$-Devil moves by burning some vertex $y$ in
$G$.  The imagined $H$-Devil burns the corresponding vertex $\phi(y)$.
The imagined $H$-Angel has a winning response $w$ for this move.

Since $w$ must be an out-neighbor of the current vertex $\phi(v)$ in $H$, the
robustness of $\phi$ guarantees a vertex $z$ in $G$ such that $vz\in E(G)$ and
$\phi(z)=w$.  The vertex $z$ cannot previously have been burned by the
$G$-Devil, because the imagined $H$-Devil would have immediately burned $w$ in
$H$ to copy that move.  Since $z$ is available and $vz\in E(G)$, the $G$-Angel
can move to $z$.  This preserves the property that the $H$-Angel is on the
image of the position of the $G$-Angel, and the game continues.  Since the
$H$-Angel can move forever, the $G$-Angel also can move forever.
\end{proof}

\section{Walker-Blocker on High-dimensional Wedges}\label{wedge}
In this section we prove that Walker wins the prefix game on wedge posets of
dimension at least 14.  The \emph{$d$-dimensional wedge}, written $\ved$, is
the cover digraph of the infinite wedge poset in $d$ dimensions.  The vertices
are nonnegative integer $d$-tuples, with $xy\in E(\ved)$ if $y$ is obtained
from $x$ by increasing one coordinate by $1$.  The root is $(0,\ldots,0)$.

We compare the wedge with the ``power-2'' Angel-Devil game on $\ints^2$.
An Angel of power $k$ can move to any unburned square that is at most $k$ units
away in each horizontal or vertical direction.  Thus in the digraph for the
power-1 game each vertex has outdegree 8, while in the power-2 game the
vertices have outdegree 24.  It is known that Devil wins the power-1
game (see Conway~\cite{Conway}), while Angel wins the power-2 game (proved
independently by Kloster~\cite{Kloster} and M\'ath\'e~\cite{Mathe}).

To prove our result, we give a robust map from $\vey{24}$ to the digraph for
the power-2 Angel on $\ints^2$.  Since Angel wins that game,
Theorem~\ref{thm:mapwin} implies that Angel wins on $\ved$ when $d\ge 24$ (a
refinement of the argument lowers the bound to $d\ge 14$).  By
Theorem~\ref{MB-AD_equiv}, Walker then wins the prefix game on $\ved$, hitting
all levels.

The construction of our robust map uses the following intuitive idea: if
Angel has $d$ different ``types of move'' in some digraph, and these moves
commute, then we can introduce a (highly redundant) coordinate system on the
graph by counting how many times Angel has made each type of move.  This
coordinate system induces a robust map from $\ved$ into the digraph.  The
following lemma formalizes the idea.

\begin{lemma}\label{lem:movetrick}
If $H$ is a rooted digraph with $V(H) \subset \ints^n$, and $M \subset \ints^n$
is a finite set such that $xy \in E(H)$ implies $y-x\in M$, then there is
a robust map from $\vey{\sizeof{M}}$ to $H$.
\end{lemma}
\begin{proof}
Let $d = \sizeof{M}$, and let $\VEC m1d$ be the elements of $M$.
Define $\phi\st\ved\to V(H)$ by $\phi(\VEC x1d) = h_0 +\SE i1d x_im_i$,
where $h_0$ is the root of $H$.  Since $\phi(0,\ldots,0) = h_0$, the start
condition is satisfied.  Now consider $(\phi(v),w) \in E(H)$.  We must find
$z\in V(\ved)$ such that $\phi(z)=w$ and $vz\in E(G)$.  Since
$(\phi(x),v)\in E(H)$, the hypothesis guarantees existence of $m_i \in M$ such
that $\phi(x)+m_i = v$.  With $e_i$ denoting the unit vector with $1$ in
coordinate $i$, we have $\phi(x+e_i)=\phi(x)+m_i = v$, and
$(x,x+e_i)\in E(\ved)$.  Hence $\phi$ is robust.
\end{proof}

The underlying digraphs of the classical Angel-Devil game fit the hypothesis of
the lemma, yielding the following corollary:

\begin{corollary} \label{maincor}  
Angel wins the Angel-Devil game on $\vey{d}$ for $d\ge 14$ (and hence also
Walker wins the prefix game).
\end{corollary}
\begin{proof}
Since Angel wins the power-2 Angel-Devil game~\cite{Kloster,Mathe}, in 
which Angel always has 24 possible moves expressed as coordinate vectors,
Lemma~\ref{lem:movetrick} and Theorem~\ref{thm:mapwin} together imply that
Angel wins in $\vey{24}$.  Furthermore, Angel can win that power-2 Angel-Devil
game using only moves changing the horizontal coordinate by at most 2 and
the vertical coordinate by at most 1 (proved by W\"astlund~\cite{Wastlund});
hence Angel wins in $\vey{14}$ (and thus in all higher-dimensional wedges).
\end{proof}

In Section~2, we showed that Breaker wins the ordered chain game on the wedge
$\vey{2}$.  The question remains: Who wins when $3\le d\le 13$?  We conjecture
the following.
\begin{conj} \label{bigconj}
For $d\ge 3$, Walker wins the prefix game on $\ved$.
\end{conj}

\section{Fractional Devils and Biased Games}\label{bias}
In this section we consider {\it $b$-biased games}, where Blocker makes $b$
moves in response to every move by Walker.  Similarly, Devil burns $b$
positions on each move in the $b$-biased Angel-Devil game.  To study
$b$-biased games, we introduce another variation.

For a positive real number $p$, the \emph{fractional $p$-Devil game} is
played like the Angel-Devil game, but now vertices are not simply burned or
unburned.  Instead, each vertex has \emph{damage} between $0$ and $1$,
initially $0$.  Angel may move to a vertex if its damage is less than $1$.
Devil, on his turn, increases the total damage by at most $p$ on a finite set
of vertices (the damage on any one vertex cannot decrease).

In the fractional $1$-Devil game, Devil may choose to burn one vertex per turn
as in the standard Angel-Devil game, but he may also choose to burn several
vertices partially.  Thus the fractional Devil is at least as strong as the
standard Devil.  We do not know whether there are any digraphs on which the
fractional $1$-Devil wins but the standard Devil loses.

To apply the fractional model, we also generalize the concept of robustness:

\begin{definition}
Let $G$ and $H$ be digraphs with roots $g_0$ and $h_0$.  For $k\in\NN$, a
\emph{$k$-robust map} from $G$ to $H$ is a map $\phi\st V(G) \to V(H)$ with
$\phi(g_0) = h_0$ such that when there is an edge from $\phi(v)$ to $w$ in $H$,
there are at least $k$ vertices $z$ in $G$ with $\phi(z)=w$ and $vz\in E(G)$.
\end{definition}

We can now generalize Theorem~\ref{thm:mapwin}:

\begin{theorem}\label{kmap}
Let $G$ and $H$ be two rooted digraphs, let $p$ be a positive real number, and
let $\phi$ be a $k$-robust map from $G$ to $H$.  If Angel wins the fractional
$p$-Devil game in $H$, then Angel wins the fractional $pk$-Devil game in $G$.
\end{theorem}
\begin{proof}
We modify the proof of Theorem~\ref{thm:mapwin} in a straightforward way.
When the $G$-Angel is at $v$, and the $G$-devil adds damage $(\VEC x1r)$ to
vertices $(\VEC v1r)$ (totalling at most $pk$), the simulated $H$-devil adds
damage $x_i/k$  to vertex $\phi(v_i)$, for $1\le i\le r$; this is a legal move.
Let moving from $\phi(v)$ to $w$ be the response in the winning strategy for
the imaginary $H$-angel; this requires that damage less than $1$ has been done
to $w$.  Since damage has been done via $\phi$, we conclude that the total
damage by the $G$-Devil on preimages of $w$ has been less than $k$.
Since $\phi$ is $k$-robust, there are at least $k$ such preimages that are
out-neighbors of $v$ in $G$, and hence one of them is available as a move
for Angel in $G$.
\end{proof}

Kloster's proof~\cite{Kloster} that Angel wins the power-2 Angel-Devil game can
be adapted to show that Angel also wins that game against the fractional
$1$-Devil.  Hence Angel wins the fractional $1$-Devil game in $\vey{24}$, by
Lemma~\ref{lem:movetrick} and Theorem~\ref{kmap}.  We transform this winning
strategy into a winning strategy against the (biased) fractional $b$-Devil in
$\vey{24b}$ by constructing highly robust maps and applying Theorem~\ref{kmap}
again.
\begin{lemma}\label{lem:mult}
For $d,k\in\NN$, there is a $k$-robust map from $\vey{kd}$ to $\vey{d}$.
\end{lemma}
\begin{proof}
Define $\phi$ by mapping $(\VEC x1{kd})$ to the $d$-tuple whose $i$th
entry is $\SE j0{k-1} x_{jd+i}$.  To see that $\phi$ is $k$-robust, consider
$(\phi(v),w)\in E(\vey{d})$.  Note that $w=\phi(v)+e_h$ for some $h$, where
$e_h$ is $1$ at index $h$ and $0$ everywhere else.  The $k$ vertices of the
form $v+e_{h+jd}$, where the subscript is taken modulo $kd$, are all
outneighbors of $v$ in $\vey{kd}$ that map to $w$. 
\end{proof}

Note that every partition of the $kd$ coordinates into $d$ classes of size $k$
induces a $k$-robust map; let the $i$th coordinate of the image be the sum of
the $k$ coordinates in the $i$th class of the partition.

\begin{theorem}
Angel wins the fractional $b$-Devil game in $\vey{24b}$ and hence also
wins the $b$-biased Angel-Devil game in $\vey{24b}$.
\end{theorem}
\begin{proof}
Lemma~\ref{lem:mult} provides a $b$-robust map from $\vey{24b}$ to $\vey{24}$.
Since Kloster's proof~\cite{Kloster} implies that Angel wins the fractional
$1$-Devil game in $\vey{24}$, Theorem~\ref{kmap} implies that Angel wins
against the fractional $b$-Devil in $\vey{24b}$.  Since the fractional
$b$-Devil is at least as strong as the Devil of bias $b$, Angel wins the
$b$-biased game in $\vey{24b}$.
\end{proof}

There are easy relationships between some biased games on wedges.
\begin{prop}\label{gap}
If Blocker wins the $b$-biased prefix game on $\vey{d}$, then 
Blocker wins the $(b+1)$-biased prefix game on $\vey{d+1}$.
\end{prop}
\begin{proof}
On each turn Blocker plays the outneighbor of Walker's previous move with
coordinate $d+1$ increased and uses the remaining $b$ moves to play the
strategy for $d$ dimensions.  Any move by Walker that increases the last
coordinate immediately loses a level, so Walker does best by playing the game
in $d$ dimensions.
\end{proof}

Define the {\it gap} of a biased Walker-Blocker game to be $d-b$, where $d$ is
the dimension and $b$ is the bias.
This suggests a natural question: Given a gap $g$, what is the smallest
dimension $d$ such that gap $g$ suffices for Blocker to hold Walker to a
fraction of the levels strictly less than 1?  Let $d(g)$ denote this minimum
dimension.  We know that $d(0)=1$, $d(1)=2$, and $3\le d(2)\le 5$.  The fact
that $d(1)=2$ follows from Theorem~\ref{wedge2}, since gap 1 in two dimensions
is an unbiased game.  The other results are trivial, except for the fact that
$d(2)\le 5$, which we prove below.

\begin{prop}
With bias 3, Blocker holds Walker to at most $4/5$ of the levels in the
ordered chain game on $\vey{5}$, and hence $d(2)\le5$.
\end{prop}
\begin{proof}
We may assume that Walker first takes 00000.  Blocker forces Walker to 
increase a coordinate that is 0 on each subsequent move to avoid skipping
levels, eventually forcing Walker to skip a level.  The table below gives
the moves by Walker and Blocker, up to symmetry.   With the three moves of his
first turn and two moves of his second, Blocker occupies all points at the
fifth level with largest coordinate $1$.  Blocker's other moves prevent
Walker from moving up one level to reach a point having a coordinate larger
than $1$.  After at most four moves,  Walker must skip a level, and the pattern
repeats.  Thus Blocker can hold Walker to $4/5$ of the levels.
\end{proof}

\begin{tabular}{l|lll}
Walker & \multicolumn{3}{c}{Blocker} \\
\hline
00000 &   01111 & 10111 & 11011 \\
10000 &   20000 & 11101 & 11110 \\
11000 &   21000 & 12000 & -- \\
11100 &   21100 & 12100 & 11200
\end{tabular}
\bigskip

We conclude with Conjecture~\ref{biggerconj}, which strengthens Conjecture~\ref{bigconj}.

\begin{conj}
\label{biggerconj}
There exists a constant $k$ such that for every dimension $d$, 
Maker wins the $(d-k)$-biased ordered chain game on $\ved$. 
\end{conj}
\bibliographystyle{amsplain}

\end{document}